\documentclass[onecolumn]{article} 
\usepackage{float}
\usepackage{amssymb}
\usepackage{amsmath}
\usepackage{amsthm}
\usepackage{caption}
\usepackage{stfloats} 
\usepackage{comment}
\usepackage[ruled]{algorithm2e}
\usepackage{float}
\usepackage{multicol}
\usepackage{color}
\newcounter{multifig}
\usepackage{thmtools,thm-restate}
\usepackage{multirow}
\usepackage{mathtools, cuted}

\newcommand {\R}{\mathbb{R}}

\usepackage{lipsum}

\usepackage{array}
\usepackage{makecell}

\usepackage{stackengine}
\usepackage{bm}
\usepackage[utf8]{inputenc}

\newcommand\norm[1]{\left\lVert #1 \right\rVert} 

\title{The Convergence of Finite-Averaging of AIMD for Distributed Heterogeneous Resource Allocations}
\author{Syed Eqbal Alam \\ CIISE,
		Concordia University, \\
		Montreal, Quebec, Canada,
		\\
\and Fabian Wirth \\ Faculty of Computer Science and Mathematics,\\ University of Passau, Passau, Germany,\\
\and Jia Yuan Yu \\ CIISE,
	Concordia University, \\
	Montreal, Quebec, Canada, \\
	\and Robert Shorten\\ Dyson School of Design Engineering \\ Imperial College London, UK}
\date{September 2019}	
 
\begin{document}

\maketitle
\newtheorem{thm}{Theorem}
\newtheorem{lem}{Lemma}
\newtheorem{propos}{Proposition}
\newtheorem{corol}{Corollary}
\newtheorem{prob}{Problem}
\newtheorem{defini}{Definition}
\newtheorem{assump}{Assumption}
\begin{abstract}
In several social choice problems, agents collectively make decisions over the allocation of multiple divisible and heterogeneous resources with capacity constraints to maximize utilitarian social welfare. The agents are constrained through computational or communication resources or privacy considerations. In this paper, we analyze the convergence of a recently proposed distributed solution that allocates such resources to agents with minimal communication. It is based on the randomized additive-increase and multiplicative-decrease (AIMD) algorithm. The agents are not required to exchange information with each other, but little with a central agent that keeps track of the aggregate resource allocated at a time. We formulate the time-averaged allocations over finite window size and model the system as a Markov chain with place-dependent probabilities. Furthermore, we show that the time-averaged allocations vector converges to a unique invariant measure, and also, the ergodic property holds. 
\end{abstract}

\section{Introduction}
Recently, the social choice theory has attracted significant attention from the artificial intelligence community \cite{Brandt2016}, \cite{Freeman2017}, \cite{Benade2019}. In many social choice problems, agents collectively make decisions over the allocation of multiple divisible and heterogeneous resources. In this paper, we take the utilitarian viewpoint \cite{Boutilier2012} in which multiple agents coordinate to minimize the sum of their costs subject to capacity constraints on the resources, called {\em budgeted social choice} \cite{Lu2011}.  It is challenging to obtain optimal allocations that solve such problems with multiple resources; furthermore, the agents may have limited computation capability, and may not wish to communicate their costs or allocations because of privacy reasons. To this end, a communication-efficient, iterative, and randomized algorithm \cite{Syed2018} is proposed to minimize the social cost with equality capacity constraints (as stated in Problem \ref{obj_fn1}) in a distributed way. For $m$ resources, at most $m$ bits of information is exchanged per iteration. In the algorithm, the agents do not exchange their allocations or costs with each other, but communicate a little with a central agent that keeps track of the aggregate allocations; however, it does not have access to the individual allocations or costs of agents. Additionally, we use a complicated notion of time steps that the time between two steps is a random variable. Notice that the proposed algorithm is Pareto optimal in the sense that an agent can not reduce its cost without increasing the cost of another agent; also, it is budget balanced. 
Furthermore, the algorithm is privacy-preserving and not a direct revelation mechanism. Also, it is not a strategyproof mechanism---an agent achieves best allocations by being truthful irrespective of what other agents do. Additionally, it is not a random dictatorship mechanism \cite{Gibbard1977}---which is the only strategyproof and ex-post Pareto efficient mechanism \cite{Gross2017}. Pareto optimality for resource allocation with connectivity constraints is recently studied in \cite{Igarashi2019}; furthermore, a fair allocation of resources with reduced communication is studied in \cite{Oh2019}. The proof of convergence of the proposed algorithm \cite{Syed2018} was an open problem; in this paper, we provide that. 

 For a brief background, the proposed algorithm is based on a randomized {\em additive-increase and multiplicative-decrease (AIMD)} algorithm to solve the multiple divisible-resource allocation problem in a distributed way; it is a generalization of \cite{Wirth2014}. The algorithm does not require inter-agent communication; however, a one-bit feedback signal is required from the central agent when the aggregate demand exceeds the capacity of a resource. \cite{Chiu1989} proposed the AIMD algorithm for congestion avoidance in the transmission control protocol (TCP). The theoretical properties of the AIMD algorithm are studied extensively \cite{Dumas2002}, \cite{Srikant2004}, also in resource allocation context \cite{Avrachenkov2017}, \cite{Corless2016}. 
 The proposed algorithm consists of two phases---{\em additive increase (AI)} phase and {\em multiplicative decrease (MD)} phase. In the AI phase, agents continuously increase their demands for a shared resource until the aggregate resource demand exceeds the capacity of that resource; after which the central agent broadcasts a one-bit {\em capacity event} signal in the system, to notify the agents that aggregate demand for the resource has exceeded its capacity, this is done for all the resources in the system. After receiving this signal, an agent decides in a probabilistic way to reduce the resource demand or not; this is the MD phase. The probability distribution depends on the private cost function of the agent and its average resource allocations; by doing so, the social cost is minimized in a distributed and randomized way.

Our contribution here is to derive the matrix for multiple divisible and heterogeneous resources called {\em multi-variate AIMD matrix} for the proposed algorithm. The matrix is a non-negative column stochastic matrix with randomized entries. Additionally, we find the average allocation over a fixed window size and model the system as a Markov chain with place-dependent probabilities. Given the basic settings and mild assumptions on the probability distribution with which an agent responds to the capacity event, we show that the average allocation converges to a {\em unique invariant measure} asymptotically, and the {\em ergodic property} holds. We do so, using the average contraction techniques \cite{Barnsley1988} and ergodic property \cite{Elton1987}. Thus, we say that for large window size, the average allocations reach near-optimal values with high probability. 
\section{Problem formulation}  \label{prob_form}
Suppose that $n$ agents coordinate to access $m$ limited divisible and heterogeneous resources. The resources $\mathcal{R}^1, \mathcal{R}^2, \ldots, \mathcal{R}^m$ have capacities $\mathcal{C}^1, \mathcal{C}^2, \ldots, \mathcal{C}^m$, respectively. Each agent associates a cost to a specific allocation of the resources. 
\subsection{Notation and assumptions}
Let $\mathbb{R}$ denote the set of real numbers and $\mathbb{R}_+$ denote
the set of non-negative real numbers. The vector space of column vectors of
length $n$ with real entries is denoted by $\R^n$. In this space, the vectors $\mathbf{e}_j$ are the canonical basis vectors.
The variable $t \in \mathbb{R}_+$
represents time. For the intermittent discrete process we use the time
sets $\mathbb{N} = \{0,1,2,\ldots\}$ and $\mathbb{N}_+ = \{1,2,3, \ldots\}$. 
Agents are indexed by 
$i=1,2,\ldots,n$, whereas resources are indexed by $j=1,2,\ldots,m$. We
will abbreviate $\underline{n}  \triangleq \{ 1,\ldots,n \}$ and similarly for
$\underline {m}$.
Let
$x_i^j(t) \in [0, \mathcal{C}^j]$ denote the amount of resource
$\mathcal{R}^j$ allocated to agent $i$ at time $t$, for $i\in\underline{n}$
and $j\in\underline{m}$. Every agent $i$ has a cost function $f_i: \mathbb{R}^m \to \mathbb R_+$, which we will assume to satisfy the following standing assumption.
\begin{assump} [Cost function]
For $i \in \underline{n}$, the cost function $f_i : \mathbb{R}^m \to \mathbb
R_+$ is continuously differentiable, strictly convex, and strictly increasing in
each variable. The latter condition means that for every $j \in
\underline{m}$ and $ \mathbf{x} \in \mathbb{R}^m$ the function
\begin{equation*}
    h \mapsto f_i(\mathbf{x} + h \mathbf{e}_j)
\end{equation*}
is strictly increasing.
\end{assump}
 We aim to prove the convergence of the solution \cite{Syed2018} of the following optimization problem.
\begin{prob}[Optimization problem] \label{obj_fn1}
\begin{align*} 
\begin{split}
\underset{{y}^1_1, \ldots, {y}^m_n}{\mathrm{minimize}} \quad &\sum_{i=1}^{n} f_i(y^1_i, y^2_i,
\ldots, y^m_i),    		\\
\mathrm{subject \ to} \quad
&\sum_{i=1}^{n} y^j_i = \mathcal{C}^j, \quad \mbox{for all } j\in\underline{m},		\\
&y^j_i \geq 0, \quad \mbox{for all } i \in \underline{n}, \mbox{ and } j\in\underline{m}.
\end{split}
\end{align*}
\end{prob}
Here, the decision variables are denoted by $y^j_i$. They represent a
certain allotment of the resource $\mathcal{R}^j$ to the agent $i$, for all $i$ and $j$. There are total of $nm$ decision variables. 
 By compactness of the constraint set and continuity of the cost function $f_i$, an optimal solution exists. Additionally, by the assumption of strict convexity of the cost functions $f_i$, the optimal solution is unique.
 
For $k \in \mathbb{N}$, let $t_k$ be the discrete time step, and let $T \in \mathbb{N}_+$ be the fixed time window, where $T \leq k+1$. Also, let $\overline{\mathbf{x}}^j_T(t_k) \in \mathbb{R}^n$ be the time-averaged allocation of resource $\mathcal{R}^j$ over the time window $T$; we define $\overline{\mathbf{x}}^j_T(t_k)$ as follows
\begin{align} \label{eq:T_average}
\overline{\mathbf{x}}^j_T(t_k) & \triangleq \frac{1}{T} \sum_{\ell=0}^{T-1} \mathbf{x}^j(t_k-t_{\ell}), \text{ for } j \in \underline{m}, \text{ and } k \in \mathbb{N}. 
\end{align}
For $k \in \mathbb{N}$, let $\bm{\xi}(t_k)$ is defined as
\begin{align} \label{eq:xim}
\bm{\xi}(t_k) \triangleq [\overline{\mathbf{x}}^1_T(t_k)^\top \quad \overline{\mathbf{x}}^2_T(t_k)^\top \quad \ldots \quad \overline{\mathbf{x}}^m_T(t_k)^\top]^\top. 
\end{align}
Now, let us define a simplex as follows.
\begin{defini}[Simplex] \label{def:simplex}
Let $\Sigma$ be the simplex in $\mathbb{R}^{n}$, we define it as follows,
\begin{align}
\Sigma &\triangleq \big \{ \mathbf{y} = (y_1,y_2,\ldots,y_n) \vert \sum_{i=1}^{n} y_i=1, y_i \geq 0, \text{ for } i \in \underline{n} \big \}. 
\end{align}
\end{defini}
 Similarly, we define simplex $\Sigma^{T}$ in $\mathbb{R}^{Tn}$. We derive multi-variate AIMD matrix for the proposed distributed and randomized algorithm \cite{Syed2018}. Furthermore, we model the system as a Markov chain with place dependent probabilities; the probabilities depend on the average allocation $\overline{\mathbf{x}}^j_T(t_k)$ for finite window size $T$. Now, let $\pi^T$ be an {\em invariant probability distribution} on $\big( \Sigma^{T} \big)^m$, where $m$ is the number resources in the system. Then we aim to show that the proposed distributed and randomized algorithm achieves the following goal,  
\begin{itemize}
\item[(i)] for every finite window size $T \geq 1$, there exists a unique invariant measure $\pi^T$ on $\big( \Sigma^{T} \big)^m$, and,
\item[(ii)] for every $\bm{\xi}(t_k) \in \big( \Sigma^{T} \big)^m$ and given initial value $\bm{\xi}(t_0) \in \big( \Sigma^{T} \big)^m$, the ergodic property holds, that is,
\begin{align*}
\lim_{k \to \infty} \frac{1}{k+1} \sum_{\ell=0}^k \bm{\xi}(t_{\ell}) = \mathbb{E}(\pi^T), \text{ almost surely},
\end{align*}
where $\mathbb{E}(\cdot)$ represents the expected value.
\end{itemize}
Thus, for large window size, the average allocations will reach close to optimal values with high probability.
\section{Multi-resource allocation algorithm} \label{divisible_mul_res}
We briefly present the distributed algorithm \cite{Syed2018} for allocating multiple divisible and heterogeneous resources, it is a generalization of the single allocation algorithm \cite{Wirth2014}, which is based on stochastic AIMD algorithm. Here, each agent runs its distributed algorithm. The algorithm consists of two phases---{\em additive increase (AI)} and {\em multiplicative decrease (MD)}. In the AI phase, an agent keeps increasing its demand for a resource linearly by a positive constant called {\em additive increase factor} until the aggregate demand for the resource reaches the capacity of that resource. After which a central agent broadcasts a one-bit {\em capacity event signal} in the system. After receiving this signal, an agent responds in a probabilistic way, whether to decrease its demand for a resource by a constant called {\em multiplicative decrease factor} or not; this is the MD phase. After the MD phase, again, the AI phase starts, and agents keep increasing their resource demands linearly by additive increase factors until they receive the next capacity event signal. This process repeats over time.

 Let $\alpha^j>0$ be the additive increase factor and $0 < \beta^j < 1$ be the multiplicative decrease factor of an agent for resource $\mathcal{R}^j$, for $j \in \underline{m}$. Furthermore, let $\Omega$ be a sample space for Bernoulli trials, and $\bm{\beta}_i^j: \Omega \to \mathbb{R}$ be a random variable of agent $i$ for resource $\mathcal{R}^j$, it takes value $\beta^j$, or $1$, for $i \in \underline{n}$, and $j \in \underline{m}$. For all $j$, let $\Gamma^j$ be the {\em normalization factor} of resource $\mathcal{R}^j$ broadcasted at the start of the algorithm by the central agent of the system. Now, suppose that the multiagent system starts at time instant $t_0=0$, and every algorithm is initialized with several parameters such as $x_i^j(t_0)$, $\alpha^j$, $\beta^j$, and $\Gamma^j$, for $i \in \underline{n}$, and $j \in \underline{m}$. After time instant $t_0$, agent $i$ keeps increasing its demand for the resource linearly by $\alpha^j$ until the aggregate demand $\sum_{i=1}^n x_i^j(t)$ is equal to $\mathcal{C}^j$ at time instant $t_1^j$. At $t_1^j$ the first capacity event of resource $\mathcal{R}^j$ occurs, and the central agent broadcasts a {\em one-bit} feedback signal to notify the agents in the system to reduce their demand for the resource $\mathcal{R}^j$, for $j \in \underline{m}$, similarly $t_2^j$, and so on. Thus, for $j \in \underline{m}$ and $k \in \mathbb{N}_+$; $t_k^j$ denotes the time instant at which the $k^{\mathrm{th}}$ capacity event of resource $\mathcal{R}^j$ occurs, i.e.,
\begin{align*}
t_k^j = \inf_{t > t_{k-1}^j} \big \{ t \ \big \vert \ \sum_{i=1}^{n} x_i^j(t) = \mathcal{C}^j \big \}.   
\end{align*}  
Hence, for $k \in \mathbb{N}_+$, we write the AI phase as follows.
\begin{align} \label{Alpha_update}
x_i^j (t) &= 	x_i^j (t^j_{k-1}) + \alpha^j \times (t - t^j_{k-1}) \quad \mbox{if } t \in (t^j_{k-1}, t^j_k], \nonumber \\ &\mbox{ for } i \in \underline{n}, \text{ and } j \in \underline{m}.
\end{align}  
 For the sake of simplicity of notation, at the $k^{\mathrm{th}}$ capacity event,  we use $x_i^j(k)$ to represent $x_i^j(t_k^{j})$, and analogously for other variables, for all $j$ and $k$. Now, let $\lambda_i^j(k)$ (cf. \eqref{prob_x}) be the probability with which agent $i$ responds to the $k^{\mathrm{th}}$ capacity event then we write that $\mathbb{P} \big(\bm{\beta}_i^j(k) = \beta^j \big) = \lambda_i^j(k)$ and $\mathbb{P} \big( \bm{\beta}_i^j(k) = 1 \big) = 1 - \lambda_i^j(k)$ , for $i \in \underline{n}$; $j \in \underline{m}$; and $k \in \mathbb{N}$. 
Thus, after receiving this notification, agent $i$ responds by multiplicatively decreasing its demand by $\bm{\beta}_i^j(k) \in \{\beta^j, 1\}$ as 
 \begin{align} \label{Beta_update} 
x_i^j(t)= \bm{\beta}_i^j(k) x_i^j(t^{j}_k) \quad \mbox{if } t = t^{j+}_k,
\end{align}	
 for $i \in \underline{n}; j \in \underline{m}$; and $k \in \mathbb{N}$; this is the MD phase.
  Now, we define the time-averaged allocation as 
\begin{align} \label{average_eqn}
\overline{x}^j_i(k) \triangleq \frac{1}{k+1} \sum_{\ell=0}^k x^j_i(\ell), \quad \mbox{ for } i \in \underline{n}, \mbox{ and } j \in \underline{m}.
\end{align}
For all $i,j$, and $k$, the probability $\lambda^j_{i}(k)$ is calculated as
\begin{align} \label{prob_x}
 \lambda^j_{i}(k) = \Gamma^j  
\frac{{\nabla_j} f_i \big( \overline{x}^1_i(k), \ldots, \overline{x}^m_i(k) \big)}{\overline{x}^j_i(k)}.
\end{align}
 Note that for all $j$, the normalization factor $\Gamma^j$ is calculated by the central agent at the start of the system; an agent receives $\Gamma^j$ when it joins the system. It is used to keep the probability $\lambda_i^j(k)$ in $(0,1)$, for all $i$ and $k$. We call the probability $\lambda^j_{i}(k)$ as {\em drop probability} with which an agent multiplicatively decreases ({\em backs-off}) the demand for the resource $\mathcal{R}^j$, for $j \in \underline{m}$. After backing-off, agents can again start increasing their demands linearly until the next capacity event occurs. This process repeats over time.

\section{AIMD matrix for multiple resources} \label{AIMD_mat}
 Let $(t^j_{k+1} - t^j_k)$ be the time between the $k^{\mathrm{th}}$ and the ${(k+1)}^{\mathrm{th}}$ capacity events, for $k \in \mathbb{N}$. We proceed as follows to derive the AIMD matrix for multiple variables and model the system as a Markov chain with place dependent probabilities.  In \eqref{der_stoc_mat_1}, $x_i^j(k+1)$ represents the allocation at the $(k+1)^{\mathrm{th}}$ capacity event, and it is written as
\begin{align} \label{der_stoc_mat_1}
x_i^j(k+1) = \bm{\beta}_i^j(k) x_i^j(k) + \alpha^j \times (t^j_{k+1} - t^j_k), 
\end{align}  
 where $\mathbb{P} \big (\bm{\beta}_i^j(k) = \beta^j \big) = \lambda_i^j(k)$ and $\mathbb{P} \big (\bm{\beta}_i^j(k) = 1 \big) = 1- \lambda_i^j(k)$, for $i \in \underline{n}$; $j \in \underline{m}$; and $k \in \mathbb{N}$.

 With little algebraic manipulation and using the fact that $\sum_{i=1}^n x_i^j(k+1) = \sum_{i=1}^n x_i^j(k) = \mathcal{C}^j$, we find the value of $(t^j_{k+1} - t^j_k)$. Note that the time between two capacity events is a random variable.
Let $\mathbf{x}^j = [x_1^j \ x_2^j \ \ldots \ x_n^j ]^\top$, and let $\mathbf{e}^\top = [1 \quad 1 \quad \ldots \quad 1]$, where $\mathbf{e} \in \mathbb{R}^n$. Then after replacing the value of $(t^j_{k+1} - t^j_k)$ in \eqref{der_stoc_mat_1}, we obtain
\begin{align} \label{def:stoc_matrix}
\mathbf{x}^j(k+1) =
\left (\begin{bmatrix}
\bm{\beta}_1^j(k) & \ldots & 0&\ldots&&0\\
0& \bm{\beta}_2^j(k)& \ldots & 0 &\ldots&0\\
\vdots & &  \ddots & & &\vdots \\
\vdots & &   & & \ddots &\vdots \\
0  &  & \ldots  &   & 0& \bm{\beta}_n^j(k) \\
\end{bmatrix} 
+  \frac{1}{n} \mathbf{e}
\begin{bmatrix}
  1-\bm{\beta}_1^j(k) &
   \hdots &
  1-\bm{\beta}_n^j(k)
\end{bmatrix} \right) \mathbf{x}^j(k)
\end{align}
Thus, we write
\begin{align} \label{MC_Ax}
\mathbf{x}^j(k+1) = \mathbf{A}^j(k) \mathbf{x}^j(k),\quad  \text{for } j \in \underline{m}, \text{ and } k \in \mathbb{N},
\end{align}
where $\mathbf{A}^j(k) \in \mathbb{R}^{n \times n}$ is the matrix for resource $\mathcal{R}^j$; $\mathbf{A}^j(k)$ is called multi-variate {\em AIMD matrix} with randomized entries; it is a non-negative column stochastic matrix. 
 Also, let $\tilde{\bm{\beta}}^j=[\bm{\beta}_1^j \quad \bm{\beta}_2^j \quad \ldots \quad \bm{\beta}_n^j]^\top$, where $\bm{\beta}_i^j \in \{ \beta^j,1 \}$, for $i \in \underline{n}$ and $j \in \underline{m}$. Let $\textnormal{diag}(\tilde{\bm{\beta}}^j)$ be the diagonal matrix with $\bm{\beta}_i^j \in \{ \beta^j,1 \}$ as a diagonal, for $i \in \underline{n}$, and $j \in \underline{m}$. Then, for $j \in \underline{m}$ and $k \in \mathbb{N}$, we define matrix $A^j$ as
\begin{align} \label{def:matrix_A}
A^j \triangleq \textnormal{diag} (\tilde{\bm{\beta}}^j) + \frac{1}{n} \mathbf{e} (\mathbf{e}^\top -  \tilde{\bm{\beta}^j}^\top),
\end{align}
which is a non-negative column stochastic matrix. 
 Now, for $j \in \underline{m}$, let $\Upsilon^j$ be the set defined as follows:
\begin{align*}
\Upsilon^j \triangleq \big \{ \tilde{\bm{\beta}} = (\bm{\beta}_1^j, \ldots, \bm{\beta}_n^j) \in \mathbb{R}^n \vert \bm{\beta}_i^j \in \{ \beta^j, 1\},  \text{ for } i \in \underline{n} \big \}.
\end{align*}
Additionally, for $j \in \underline{m}$, let $\mathcal{F}^j$ be the set of all possible AIMD matrices $A^j$ of resource $\mathcal{R}^j$. Note that to represent a generalized version of a matrix, we drop the capacity event index $k$ from that matrix. Also, notice that for $n$ agents in the system, the set $\mathcal{F}^j$ consists of $2^n$ AIMD matrices of resource $\mathcal{R}^j$, for $j \in \underline{m}$.
 We define the set $\mathcal{F}^j$ as follows:
\begin{align*} 
\mathcal{F}^j \triangleq \big \{ \textnormal{diag}(\tilde{\bm{\beta}^j}) + \frac{1}{n} \mathbf{e} (\mathbf{e}^\top -  \tilde{\bm{\beta}^j}^\top) \vert \tilde{\bm{\beta}^j} \in \Upsilon^j \big \},\ \text{for } j \in \underline{m}.
\end{align*}
 Recall that the drop probability $\lambda_i^j(\cdot)$ is the probability with which agent $i$ responds to the capacity event of resource $\mathcal{R}^j$, for  $i \in \underline{n}$, and $j \in \underline{m}$. For $A^j \in \mathcal{F}^j$, we obtain the following probability:
\begin{align} \label{prod_prob}
\mathbb{P} \big(\mathbf{A}^j(k) = A^j \big) = \prod_{\bm{\beta}_i^j = \beta^j} \lambda_i^j(k) \prod_{\bm{\beta}_i^j = 1} (1- \lambda_i^j(k)), 
\end{align} 
we assume that the probability $\lambda_i^j(k)$ is independent, for  $i \in \underline{n}; j \in \underline{m}; \text{ and } k \in \mathbb{N}$. 

\section{Convergence analysis} \label{Conv_anal}
In this section, we assume the case where each agent considers its average allocation over the last $T \in \mathbb{N}_+$ capacity events to determine the probabilities of the back-off at the next capacity event, recall that $T$ is the fixed time window. For the sake of simplicity of analysis, we consider two resources $\mathcal{R}^1$ and $\mathcal{R}^2$, but the model can easily be extended to $m$ arbitrary resources. 
 Without loss of generality, we consider that if capacity events of the resource $\mathcal{R}^1$ occur then the evolution of its allocation is independent of the  evolution of allocation of the resource $\mathcal{R}^2$, except for the fact that the back-off probabilities depend on the average allocations of both the resources.

Recall that $\overline{\mathbf{x}}^j_T(k) \in \mathbb{R}^n$ (defined in \eqref{eq:T_average}) is the average allocation of resource $\mathcal{R}^j$ over time window $T$.
To model the evolution, it is convenient to introduce the vector of past averages as follows, for  $j=1,2$, and $k \in \mathbb{N}$,
\begin{align*}
\mathbf{x}^j_T(k)  &\triangleq \big [\mathbf{x}^j(k)^\top \quad \frac{1}{2} \big (\mathbf{x}^j(k) + \mathbf{x}^j(k-1)\big)^\top \ldots \\ & \frac{1}{T} \big(\mathbf{x}^j(k) + \mathbf{x}^j(k-1) + \ldots + \mathbf{x}^j(k-T+1) \big)^\top \big]^\top.
\end{align*}
For $j=1,2$, and $k \in \mathbb{N}$, let $\mathbf{D}^j(k) \in \mathbb{R}^{Tn \times Tn}$ be the matrix represented as follows. 
\begin{align} \label{mat_D}
  \mathbf{D}^j(k) = 
&\begin{bmatrix}
  \mathbf{A}^j(k) & 0 & 0 &\ldots & 0  \\
  \frac{1}{2} \big( \mathbf{A}^j(k) + I \big)  & 0 & 0 & \ldots  & \vdots \\
\frac{1}{3} \mathbf{A}^j(k)  & \frac{2}{3}I & 0 & \ldots &\\
  \vdots & 0 &  \ddots & \ddots & \vdots \\
    & \vdots &  & \ddots   & \\
\frac{1}{T} \mathbf{A}^j(k)  &  0& 0 & \ldots \frac{T-1}{T}I & 0
\end{bmatrix}.
\end{align}
 Let $\mathcal{Q}^j$ be the set of all $D^j$ matrices. Now, for $j=1,2$, and $k \in \mathbb{N}$, at the $(k+1)^{\mathrm{th}}$ capacity event, we have the following update for resource $\mathcal{R}^j$,
\begin{align*}
\mathbf{x}_T^j(k+1) = \mathbf{D}^j(k) \mathbf{x}_T^j(k).
\end{align*}
For $k \in \mathbb{N}$, let $K^1$ denote the set of capacity event time-instances $t_k^1$ of resource $\mathcal{R}^1$, analogously $K^2$ is defined. Suppose that a union of all the capacity event time-instances of $K^1$ and $K^2$ is taken and they are ordered in an increasing fashion, let this set be denoted by $K$, that is, $ K \triangleq K^1 \cup K^2$. Now, for $k \in \mathbb{N}$, let $t_k \in K$ be the time of occurrence of the $k^{\mathrm{th}}$ capacity event (combined). We rewrite the state vector $\bm{\xi}(k)$ (cf. \eqref{eq:xim}) for two resources as follows 
\begin{align} \label{def_xi}
\bm{\xi}(k) = [ \mathbf{x}_T^1(k)^\top \quad \mathbf{x}_T^2(k)^\top]^\top, \text{ for } k \in \mathbb{N}.
\end{align}
Now, let matrix $\mathbf{U}(k) \in \mathbb{R}^{2Tn \times 2Tn}$ be defined as
\begin{align}  \label{def:mat_U} 
\mathbf{U}(k) \triangleq \left\{
\begin{array}{ll}
\begin{bmatrix}
  \mathbf{D}^1(k) & 0  \\
 0 & I 
\end{bmatrix} & \mbox{ if } t_k \in K^1,  \vspace{0.1in}\\
 \begin{bmatrix}
  I & 0  \\
 0 & \mathbf{D}^2(k)
\end{bmatrix} & \mbox{ if } t_k \in K^2.\\
\end{array}
\right.
\end{align}
 Therefore, based on the occurrence of the capacity event of a particular resource, the matrix $\mathbf{U}(k)$ is chosen, and the corresponding state vectors are updated, but state vectors of other resource remain unchanged. Let $\mathcal{S}$ be the set of all $U$ matrices, and let $U_\sigma \in \mathcal{S}$. Then, we obtain the following Markov chain:
\begin{align} \label{Markov1}
  \bm{\xi}(k+1) =  \mathbf{U}(k) \bm{\xi}(k), \quad \text{for } k \in \mathbb{N},
\end{align}
    with place-dependent probabilities as follows:
  \begin{align*}
  \mathbb{P}(\mathbf{U}(k) &= U_\sigma) = p_U(\overline{\mathbf{x}}_{T}^1(k),\overline{\mathbf{x}}_{T}^2(k)), \text{for } k \in \mathbb{N}. 
  \end{align*}
  Here, $\overline{\mathbf{x}}^j_{T}(k) = \frac{1}{T} \big(\mathbf{x}^j(k) + \mathbf{x}^j(k-1) +\ldots + \mathbf{x}^j(k-T+1) \big)$ represents the $T^{\mathrm{th}}$ element of vector $\mathbf{x}^j_T(k)$, for $j=1,2$. Notice that $\overline{\mathbf{x}}^1_{T}(k)$ is the average allocation of resource $\mathcal{R}^1$ over the time interval $[k-T+1, \ldots, k]$, analogously, $\overline{\mathbf{x}}^2_{T}(k)$. The probability $p_U(\cdot)$ depends on the average allocation of all resources over the interval $[k-T+1, \ldots, k]$ for a particular realization. We rewrite \eqref{Markov1} as follows:
\begin{align} \label{Markov2}
  \bm{\xi}(k+1) &=  \mathbf{U}(k) \ldots \mathbf{U}(k-T+1) \bm{\xi}(k-T+1), \text{where } T \leq k+1, \text{ for } k \in \mathbb{N}.
 \end{align}
Now, we show that the matrices $U \in \mathcal{S}$ are non-expansive with respect to a suitable norm. The norms are defined as follows. 
\begin{defini} [Norms]
\begin{itemize}
\item[(i)] For $\ell = 1, 2, \ldots, T$, let $\mathbf{z}_{\ell} \in \mathbb{R}^n$,
 and
  $\mathbf{z} = [ {\mathbf{z}_1}^\top \ {\mathbf{z}_2}^\top \ \hdots \ {\mathbf{z}_T}^\top]^\top$. 
  
  Then $ \norm{\mathbf{z}}_T \triangleq \max_{\ell = 1, 2, \ldots, T} \norm{\mathbf{z}_\ell}_1, \mbox{ where } \norm{\mathbf{z}_{\ell}}_1 \triangleq \sum_{i=1}^n \vert z_{\ell i} \vert$.
\item[(ii)] For $j = 1, 2$;
$\mathbf{y}^j \in \mathbb{R}^{Tn}; \text{ and } 
\mathbf{y}= [  {\mathbf{y}^1}^\top \ {\mathbf{y}^2}^\top ]^\top.$

 Then $\norm{\mathbf{y}} \triangleq \max \{ \norm{\mathbf{y}^1}_T, \norm{\mathbf{y}^2}_T \}$.
\end{itemize}
\end{defini}
\begin{defini} [Invariant subspace]
For $j=1,2$, let $W^j$ be a subspace of $\mathbb{R}^{Tn}$. If $D^j W^j \subset W^j$, then $W^j$ is called an invariant subspace under all matrices $D^j \in \mathcal{Q}^j$.
 \end{defini}
Now, let $\bm{\beta}^j=[\beta^j \  \ldots \ \beta^j]^\top \in \mathbb{R}^n$; furthermore, for $j=1,2,$ and $k \in \mathbb{N}$, suppose that $B^j$ be the matrices where all agents back-off in the sense that multiplicative decrease factor $\tilde{\bm{\beta}}^j = \bm{\beta}^j$. Thus, $B^j$ are column stochastic matrices with positive entries. Similar to \eqref{def:matrix_A}, for $j = 1,2$, and $k \in \mathbb{N}$, we define matrix $B^j \in \mathcal{F}^j$, as follows:
\begin{align} \label{def:matrix_B}
B^j \triangleq \textnormal{diag}(\bm{\beta}^j) + \frac{1}{n} \mathbf{e}(\mathbf{e}^\top -  {\bm{\beta}^{j}}^\top).
\end{align}
For $j=1, 2$, let matrix $E^j \in \mathcal{Q}^j$ consists of the matrix $B^j \in \mathcal{F}^j$. We write $\mathbf{E}^j(k)$, similar to $\mathbf{D}^j(k)$ (cf. \eqref{mat_D}) by replacing $\mathbf{A}^j(k)$ with $\mathbf{B}^j(k)$. Thus, we have.
\begin{lem}[Contraction \cite{Wirth2014}] \label{lem:contraction} 
 For $j=1,2$, let $\mathbf{z}^j_\ell \in \mathbb{R}^n$, where $\ell = 1,2,\ldots,T$, and let $\mathbf{z}^j = [\mathbf{z}^{j\top}_{1} \ \mathbf{z}^{j\top}_{2} \ \ldots \ \mathbf{z}^{j\top}_{T} ]^\top$.
 \begin{itemize}
\item[(i)] For all matrices $D^j \in \mathcal{Q}^j$, the non-expansive property $\norm{D^j \mathbf{z}^j}_T \leq \norm{\mathbf{z}^j}_T$ holds. 
\item[(ii)]  Let $W^j$ be the subspace defined as
\begin{align*}
W^j \triangleq \{\mathbf{z}^j \in \mathbb{R}^{Tn} \quad \vert \quad \mathbf{e}^\top \mathbf{z}^j_{t} = 0,  \text{ for } t=1, 2, \ldots, T \}.
\end{align*}
Then $W^j$ is invariant under all $D^j \in \mathcal{Q}^j$.
\item[(iii)] For all matrices $E^j \in \mathcal{Q}^j$ and non-zero $\mathbf{z}^j \in W^j$, the contraction property $\norm{E^j \mathbf{z}^j}_T < \norm{\mathbf{z}^j}_T$ holds.
\end{itemize}
\end{lem}
\begin{proof}
\begin{itemize}
\item[(i)] For $j=1,2$, since $D^j \in \mathcal{Q}^j$, we obtain the following norm,
\begin{align} \label{eq:norm_10}
  \norm{D^j \mathbf{z}^j}_T =& 
\norm{\begin{bmatrix}
  A^j & 0 & 0 &\ldots & 0 \\
  \frac{1}{2} \big( A^j + I \big)  & 0 & 0 & \ldots  & \vdots \\
\frac{1}{3} A^j  & \frac{2}{3}I & 0 & \ldots &\\
  \vdots & 0 &  \ddots & \ddots & \vdots \\
   \vdots &  &  & \ddots   & \\
\frac{1}{T} A^j  &  0& 0 & \ldots \frac{T-1}{T}I & 0 \\
\end{bmatrix}
\begin{bmatrix}
  \mathbf{z}_{1}^j  \\
  \mathbf{z}_{2}^j \\
  \vdots\\
\mathbf{z}_{T}^j
\end{bmatrix}}_T \nonumber\\
= &\max \Big \{\norm{A^j \mathbf{z}_{1}^j}_1, \frac{1}{2} \norm{A^j \mathbf{z}_{1}^j + \mathbf{z}_{1}^j}_1,  \norm{ \frac{1}{3} A^j \mathbf{z}_{1}^j + \frac{2}{3} \mathbf{z}_{2}^j}_1, \ldots,  \nonumber \\ &\norm{ \frac{1}{T} A^j \mathbf{z}_{1}^j + \frac{T-1}{T} \mathbf{z}_{T-1}^j}_1 \Big \}, \quad \text{for } j=1,2.
\end{align}
For $a,b \in \mathbb{R}^n$, we know that $\norm{a + b}_1 \leq \norm{a}_1 + \norm{b}_1$. Then for $j=1,2$, we rewrite \eqref{eq:norm_10} as follows, 
\begin{align} \label{eq:norm_11}
\norm{D^j \mathbf{z}^j}_T  \leq& \max \Big \{\norm{A^j \mathbf{z}_{1}^j}_1, \frac{1}{2} \norm{A^j \mathbf{z}_{1}^j}_1 + \frac{1}{2} \norm{\mathbf{z}_{1}^j}_1, \frac{1}{3} \norm{ A^j \mathbf{z}_{1}^j}_1 + \frac{2}{3} \norm{\mathbf{z}_{2}^j}_1, \ldots,  \nonumber \\ & \frac{1}{T} \norm{  A^j \mathbf{z}_{1}^j}_1 + \frac{T-1}{T} \norm{\mathbf{z}_{T-1}^j}_1 \Big \}.
\end{align}
 Now, let the matrix $A^j \in \mathcal{F}^j$ be represented as $[a^j_{i \ell}]$, where $a^j_{i \ell} \geq 0$, for $i \in \underline{n}$; $\ell \in \underline{n}$; and $j=1,2$. Then, we write that
 \begin{align} \label{eq:l1-norm0}
 \norm{A^j \mathbf{z}_{1}^j}_1 = \sum_{i=1}^{n} \big \vert \sum_{\ell=1}^{n}(a_{i\ell}^j z^j_{1 \ell}) \big \vert, \quad \text{for } j=1,2. 
   \end{align}
   Since
   \begin{align} \label{eq:l1-norm01}
 \sum_{i=1}^{n} \big \vert \sum_{\ell=1}^{n}(a_{i\ell}^j z^j_{1 \ell}) \big \vert  \leq \sum_{i=1}^{n}\sum_{\ell=1}^{n} \big \vert (a_{i \ell}^j z^j_{1 \ell}) \big \vert, \quad \text{for } j=1, 2,
   \end{align}
   and
   \begin{align} \label{eq:l1-norm02}
 \sum_{i=1}^{n}\sum_{\ell=1}^{n} \big \vert (a_{i \ell}^jz^j_{1 \ell}) \big \vert = \sum_{\ell=1}^{n} \big( \sum_{i=1}^{n} a_{i \ell}^j\big) \big \vert z^j_{1 \ell} \big \vert, \quad \text{for } j=1,2.
  \end{align}
 Furthermore, since $A^j \in \mathcal{F}^j$ is a column stochastic matrix; therefore, $\sum_{i=1}^{n} a_{i \ell}^j = 1$, for $j=1, 2$. 
 Hence, from \eqref{eq:l1-norm0}, \eqref{eq:l1-norm01} and \eqref{eq:l1-norm02}, and by definition of the norm $\norm{\cdot}_1$, we write that 
 \begin{align} \label{eq:l1-norm} 
 \norm{A^j \mathbf{z}_{1}^j}_1 \leq \sum_{\ell=1}^{n} \big \vert z^j_{1 \ell} \big \vert = \norm{\mathbf{z}_{1}^j}_1, \quad \text{for } j=1,2.
 \end{align}
Placing the value of \eqref{eq:l1-norm} in \eqref{eq:norm_11}, we  obtain the following result, 
\begin{align} \label{eq:norm_02}
\norm{D^j \mathbf{z}^j}_T & \leq \max \Big \{\norm{\mathbf{z}_{1}^j}_1,\frac{1}{2}\big(\norm{\mathbf{z}_{1}^j}_1 +\norm{\mathbf{z}_{1}^j}_1 \big), \frac{1}{3} \big( \norm{\mathbf{z}_{1}^j}_1 + 2 \norm{\mathbf{z}_{2}^j}_1 \big), \ldots,  \nonumber \\ &\frac{1}{T} \big ( \norm{\mathbf{z}_{1}^j}_1 + (T-1) \norm{\mathbf{z}_{T-1}^j}_1 \big) \Big \}, \quad \text{for } j=1, 2.  
\end{align}
Now, without loss of generality, let $\norm{\mathbf{z}_{t}^j}_1$ be the maximum value among all $\norm{\mathbf{z}_{v}^j}_1$, where $v \in \{1,2,\ldots,t-1,t+1,\ldots,T \}$. 
Then we rewrite \eqref{eq:norm_02} as
\begin{align} \label{eq:norm_03}
\norm{D^j \mathbf{z}^j}_T & \leq \norm{\mathbf{z}_t^j}_1, \quad \text{for } j=1,2.  
\end{align}
Additionally, using the above assumption and the definition of the norm $\norm{\cdot}_T$, we get $\norm{\mathbf{z}^j}_T = \norm{\mathbf{z}_{t}^j}_1$, for $j=1,2$. Hence, from \eqref{eq:norm_03}, we obtain $\norm{D^j \mathbf{z}^j}_T \leq \norm{\mathbf{z}^j}_T$, for $j=1,2$.
\end{itemize}
Proof of (ii) and (iii), follows \cite{Wirth2014}.
\end{proof}
Now, using the above results, we show that the matrices $U \in \mathcal{S}$ are non-expansive.
\begin{lem}[Non-expansive matrix] \label{lem:non-expansive} 
\begin{itemize}
\item[(i)] For $j=1,2$, let $\mathbf{z}^j_{t} \in \mathbb{R}^n$, where $t=1,2,\ldots,T$, and let $\mathbf{z}^j = [{\mathbf{z}^j_{1}}^\top \ \ldots \ {\mathbf{z}^j_{T}}^\top]^\top$. 
 Also, let the subspace $W$ be defined as,
\begin{align*}  
W &\triangleq \{(\mathbf{z}^1, \mathbf{z}^2) \in \mathbb{R}^{2Tn} \quad \vert \quad \mathbf{e}^\top \mathbf{z}^1_{t} = \mathbf{e}^\top \mathbf{z}^2_{t} = 0, \text{ for } t=1,2,\ldots,T\}.
\end{align*}
Then, $W$ is invariant under all matrices $U \in \mathcal{S}$.
\item[(ii)] For $\mathbf{z} \in W$, matrix $U \in \mathcal{S}$ is non-expansive on the subspace $W$ with respect to the norm $\norm{.}$, i.e., $\norm{U\mathbf{z}} \leq \norm{\mathbf{z}}$.
\end{itemize}
\end{lem}
\begin{proof}
\begin{itemize}
\item[(i)] Since $A^1 \in \mathcal{F}^1$ and $A^2 \in \mathcal{F}^2$ are non-negative column stochastic matrices; thus, $\mathbf{e}^\top A^1 = \mathbf{e}^\top A^2 = \mathbf{e}^\top$. Also, $\mathbf{e}^\top I = \mathbf{e}^\top$, for identity matrix $I \in \mathbb{R}^{n\times n}$. As $\mathbf{e}^\top \mathbf{z}^1_{t} = 0$, for $\mathbf{z}_t^1 \in \mathbb{R}^{n}$. Therefore, $\mathbf{e}^\top A^1 \mathbf{z}^1_{t} = \mathbf{e}^\top \mathbf{z}^1_{t} = 0$; analogously, as $\mathbf{e}^\top \mathbf{z}^2_{t} = 0$, for $\mathbf{z}_t^2 \in \mathbb{R}^{n}$, we have $\mathbf{e}^\top A^2 \mathbf{z}^2_{t} = \mathbf{e}^\top \mathbf{z}^2_{t} = 0$, for $t=1,2,\ldots,T$. The proof is a consequence of these.

\item[(ii)]
From Lemma \ref{lem:contraction}, we know that $D^1 \in \mathcal{Q}^1$ and $D^2 \in \mathcal{Q}^2$ are non-expansive with respect to the norm $\norm{.}_T$. Thus, without loss of generality, let us assume that the capacity event of resource $\mathcal{R}^1$ occurs, then the matrix $D^1 \in \mathcal{Q}^1$ of $\mathcal{R}^1$ will be active to update its variables, but the variables of resource $\mathcal{R}^2$ will remain unchanged. Therefore, we get the following norm, 
\begin{align} \label{non_exp}
\norm{U \mathbf{z}} &= \norm{
   \begin{bmatrix}
   D^1 & 0   \\
   0 & I\\
\end{bmatrix}
\begin{bmatrix}
  \mathbf{z}^1  \\
  \mathbf{z}^2
\end{bmatrix}
 }
=
\norm{  \begin{bmatrix}
    D^1 \mathbf{z}^1\\
\mathbf{z}^2
\end{bmatrix} 
 } \nonumber\\
 &= \max \{ \norm{D^1 \mathbf{z}^1}_T, \norm{\mathbf{z}^2}_T \}.
\end{align}
From Lemma \ref{lem:contraction}, we write $\norm{D^1 \mathbf{z}^1}_T \leq \norm{\mathbf{z}^1}_T$. Placing it in \eqref{non_exp} and by definition of $\norm{\cdot}$, we get $\norm{U \mathbf{z}} \leq  \norm{\mathbf{z}}$.
 \qedhere 
 \end{itemize} 
\end{proof}
We use the following results to show the properties (non-expansive or contraction) of product of $U \in \mathcal{S}$ matrices.
\begin{lem} \label{lem:prod_stoc} 
For $\ell \in \mathbb{N}_+$, and $j=1,2$, let $A^j \in \mathcal{F}^j$ (cf. \eqref{def:matrix_A}) is a non-negative column stochastic matrix, then $(A^j)^{\ell}$ is also a non-negative column stochastic matrix.
\end{lem}
For $j=1,2$ and $\ell \in \mathbb{N}_+$, let 
\begin{align} \label{def:Xg}
X_g^j = \left\{ \begin{array}{ll}
A^j \\
B^j
\end{array}
\right., \mbox{ for } g \in \{1, 2, \ldots, \ell\},
\end{align} 
such that matrix $X_{g^*}^j = B^j$, where $g^* \in \{1, \ldots, \ell\}$.
Hence,  $X^j_\ell \ldots X^j_1$ contains at least one $B^j \in \mathcal{F}^j$ (cf. \eqref{def:matrix_B}) which is a column stochastic matrix with positive entries. Thus
\begin{corol} \label{cor_pos}
For $\ell \in \mathbb{N}_+$, $j=1,2$, and $g \in \{1,2,\ldots,\ell \}$, let $X_g^j$ is defined as \eqref{def:Xg} then $X^j_\ell \ldots X^j_2 X^j_1$ is a column stochastic matrix with positive entries.
\end{corol}
 We obtain the following non-expansive and contraction properties of the product of column stochastic matrices. 
\begin{propos} \label{prop:prodA}
\begin{itemize} For $\ell \in \mathbb{N}_+$, and $j=1,2$,
\item[(i)] let $A^j \in \mathcal{F}^j$ (cf. \eqref{def:matrix_A}) then for $\mathbf{z}^j \in \mathbb{R}^n$, the norm $\norm{(A^j)^\ell \mathbf{z}^j}_1 \leq \norm{\mathbf{z}^j}_1$ holds. 
\item[(ii)] For $g \in \{1, 2, \ldots, \ell \}$, let $X_g^j$ is defined as \eqref{def:Xg}. Then for a non-zero vector $\mathbf{z}^j \in \mathbb{R}^n$, the  norm $\norm{X^j_\ell \ldots X^j_2 X^j_1 \mathbf{z}^j}_1 < \norm{\mathbf{z}^j}_1
$ holds.
\end{itemize}
\end{propos}
\begin{proof}
 \begin{itemize}
 Proof is an easy consequence of Lemma \ref{lem:prod_stoc} and Corollary \ref{cor_pos}. \qedhere
\end{itemize}
\end{proof}
%
Now, for $j=1,2$, let $k \in \mathbb{N}$ and $t_k \in K^j$, and suppose that $k+1$ capacity events of resource $\mathcal{R}^j$ occur between two chosen time instants. For $j=1,2$, and $k \in \mathbb{N}$, we calculate the product of $k+1$ matrices $D^j \in \mathcal{Q}^j$ in \eqref{mat_Dk}.
%
\begin{align} \label{mat_Dk}
(D^j)^{k+1} = \begin{bmatrix}
  (A^j)^{k+1} & 0 & 0  &  &  &\ldots &   &   0 \\
  \frac{1}{2} \big( (A^j)^{k+1} + (A^j)^{k} \big)  &  &  &  &  & & \\
\vdots  & \vdots & \vdots  &  & & & &  \vdots\\
\frac{1}{k+2} \big((A^j)^{k+1} + (A^j)^{k} + \ldots + A^j +I \big)  & 0  &  & \ldots &  & \ldots & &    0 \\
\frac{1}{k+3} \big((A^j)^{k+1} + (A^j)^{k} + \ldots + A^j \big)  & \frac{2}{k+3}I & 0 & \ldots  &  & \ldots &   &   \vdots\\
& 0& \ddots &  & & &  &  \\
   \vdots & \vdots &\vdots  &   \ddots &   &  &    & \vdots\\
\frac{1}{T} \big((A^j)^{k+1} + (A^j)^{k} + \ldots + A^j \big)  &  0& 0 & \ldots & \frac{T-r-1}{T}I &0 & \ldots & 0
\end{bmatrix}.
\end{align}
%
Analogously, suppose that if $k+1$ capacity events of resource $\mathcal{R}^j$ occur, during which at least once all the agents back-off then there exists at least one matrix $E^j \in \mathcal{Q}^j$ in the product term; recall that matrix $E^j \in \mathcal{Q}^j$ consists of $B^j \in \mathcal{F}^j$ (cf. \eqref{def:matrix_B}). For $j=1,2$, let
\begin{align} \label{def:Mg}
M_g^j = \left\{ \begin{array}{ll}
D^j \\
E^j
\end{array}
\right., \mbox{ for } g \in \{1, 2, \ldots, k+1 \},
\end{align}
 and there exists at least one $g^* \in \{ 1, \ldots, k+1 \}$ such that $M_{g^*}^j = E^j$, then the product of $k+1$ such matrices is written as $M_{k+1}^j \ldots M_1^j$.
 Furthermore, let $X^j_g$ is defined as \eqref{def:Xg} with the same indexes $g$ and $g^*$ as in \eqref{def:Mg}. Then, for $t_k \in K^j$, and $j = 1, 2$, we obtain the product $M_{k+1}^j \ldots M_1^j$, by replacing $(A^j)^\ell$ with $X_{\ell}^j \ldots X_1^j$ in the definition of the matrix $(D^j)^{k+1}$ (cf. \eqref{mat_Dk}), where $\ell \in \{1,\ldots, k+1\}$.
We get the following results.
\begin{lem} \label{lem:prod_cont}
For $j=1,2$, let $k \in \mathbb{N}$ be a fixed value and $t_k \in K^j$. Furthermore, let $W^j$ be an invariant subspace under matrix $(D^j)^{k+1}$ (cf. \eqref{mat_Dk}). Then, 
\begin{itemize}
\item[(i)] $\norm{(D^j)^{k+1} \mathbf{z}^j}_T \leq \norm{\mathbf{z}^j}_T$, for all $\mathbf{z}^j \in W^j$.
\item[(ii)] $\norm{M_{k+1}^j \ldots M_1^j \mathbf{z}^j}_T < \norm{\mathbf{z}^j}_T$, for all non-zero $\mathbf{z}^j \in W^j$, where $M_g^j$ is defined as \eqref{def:Mg}, and $g \in \{1, \ldots, k+1 \}$.
\end{itemize}
\end{lem}
\begin{proof}
\begin{itemize}
\item[(i)] 
For $j=1,2$, from \eqref{mat_Dk}, with little algebraic manipulation, we obtain the norm $\norm{(D^j)^{k+1} \mathbf{z}^j}_T$ as follows.
\begin{align*}
\norm{(D^j)^{k+1} \mathbf{z}^j}_T = \norm{\begin{bmatrix}
  (A^j)^{k+1} \mathbf{z}_1^j \\
  \frac{1}{2} \big( (A^j)^{k+1} + (A^j)^{k} \big) \mathbf{z}_1^j\\
\vdots\\
\frac{1}{k+2} \big((A^j)^{k+1} + (A^j)^{k} + \ldots + A^j +I \big) \mathbf{z}^j_1 \\
\frac{1}{k+3} \big((A^j)^{k+1} + (A^j)^{k} + \ldots + A^j \big) \mathbf{z}^j_1   + \frac{2}{k+3} \mathbf{z}^j_2\\
  \vdots \\
\frac{1}{T} \big((A^j)^{k+1} + (A^j)^{k} + \ldots + A^j \big) \mathbf{z}^j_1 + \frac{T-k-1}{T} \mathbf{z}^j_{T-k-1}
\end{bmatrix}
}_T. 
\end{align*}
Therefore, for all $j$, by definition of the norm $\norm{\cdot}_T$ and using the fact that, for $a_1 \in \mathbb{R}^n$ and $a_2 \in \mathbb{R}^n$, the norm $\norm{a_1 + a_2}_1 \leq \norm{a_1}_1 + \norm{a_2}_1$, we get the following result
\begin{align} \label{prod_mat_norm}
 \norm{(D^j)^{k+1} \mathbf{z}^j}_T \leq & \max \Bigg \{ \norm{(A^j)^{k+1} \mathbf{z}_1^j}_1, \frac{1}{2} \norm{ (A^j)^{k+1} \mathbf{z}_1^j}_1 + \nonumber \\ & \frac{1}{2} \norm{ (A^j)^{k} \mathbf{z}_1^j}_1, \ldots, \nonumber \\ & \frac{1}{T} \norm{(A^j)^{k+1} \mathbf{z}_1^j}_1 + \frac{1}{T} \norm{(A^j)^{k} \mathbf{z}_1^j}_1 + \ldots +  \nonumber \\ & \hspace{0.25in} \frac{1}{T} \norm{A^j \mathbf{z}_1^j}_1 + \frac{T-k-1}{T} \norm{ \mathbf{z}^j_{T-k-1}}_1 \Bigg \}. 
\end{align}
From Proposition \ref{prop:prodA}, we write that $\norm{(A^j)^{\ell} \mathbf{z}^j_1}_1 \leq \norm{\mathbf{z}^j_1}_1$, for $\ell \in \mathbb{N}_+$, where $\mathbf{z}^j_1 \in \mathbb{R}^n$, and $j=1,2$. Thus, for $j=1,2$, we rewrite \eqref{prod_mat_norm} as follows,
\begin{align} \label{prod_mat_norm2}
 \norm{(D^j)^{k+1} \mathbf{z}^j}_T \leq & \max \Bigg \{ \norm{\mathbf{z}_1^j}_1, \frac{1}{2} \Big( \norm{\mathbf{z}_1^j}_1 + \norm{\mathbf{z}_1^j}_1 \Big), \ldots,  \nonumber \\ & \frac{1}{T} \Big( (k+1)\norm{\mathbf{z}_1^j}_1 +  (T-k-1) \norm{\mathbf{z}^j_{T-k-1}}_1 \Big) \Bigg \}. 
\end{align}
Without loss of generality, suppose that $\norm{\mathbf{z}_i^j}_1$ is the maximum value among all $\norm{\mathbf{z}_t^j}_1$, where $t=1, 2, \ldots, i-1, i+1, \ldots, T$. Hence, we rewrite  \eqref{prod_mat_norm2} as follows,
\begin{align} \label{prod_mat_norm21}
 \norm{(D^j)^{k+1} \mathbf{z}^j}_T \leq \norm{\mathbf{z}_i^j}_1, \quad \text{for } j=1,2. 
\end{align}
Thus, by definition of the norm $\norm{\mathbf{z}^j}_T$, we write that $\norm{\mathbf{z}^j}_T = \norm{\mathbf{z}^j_i}_1$. After replacing this in \eqref{prod_mat_norm21}, we obtain the following result
\begin{align*}
 \norm{(D^j)^{k+1} \mathbf{z}^j}_T \leq \norm{\mathbf{z}^j}_T, \quad \text{for } j=1,2. 
\end{align*}
\item[(ii)] For $j=1,2$, we obtain the norm $\norm{M_{k+1}^j \ldots M_2^j M_1^j \mathbf{z}^j}_T$ as follows.  
\begin{align} \label{prod_mat_norm31}
 \norm{M_{k+1}^j \ldots M_2^j M_1^j \mathbf{z}^j}_T \leq & \max \Bigg \{ \norm{X^j_{k+1} X^j_{k} \ldots X^j_1  \mathbf{z}_1^j}_1, \frac{1}{2} \norm{ X^j_{k+1} X^j_{k} \ldots X^j_1 \mathbf{z}_1^j}_1 + \nonumber \\ & \frac{1}{2} \norm{X^j_{k} X^j_{k-1} \ldots X^j_1 \mathbf{z}_1^j}_1,  \ldots, \frac{1}{T} \norm{X^j_{k+1} X^j_{k} \ldots X^j_1 \mathbf{z}_1^j}_1 + \nonumber \\ & \frac{1}{T} \norm{X^j_{k} X^j_{k-1} \ldots X^j_1 \mathbf{z}_1^j}_1 +  \ldots + \nonumber \\ & \frac{1}{T} \norm{X_1^j \mathbf{z}_1^j}_1 + \frac{T-k-1}{T} \norm{\mathbf{z}^j_{T-k-1}}_1 \Bigg \}. 
\end{align}
Since, for $j=1,2$, the product $X^j_{k+1} X^j_{k} \ldots X^j_1$ consists of at least one $B^j \in \mathcal{F}^j$ which has positive entries, then from Proposition \ref{prop:prodA}, we write that $\norm{X^j_{k+1} X^j_{k} \ldots X^j_1 \mathbf{z}^j_1}_1 < \norm{\mathbf{z}^j_1}_1$, where $\mathbf{z}^j_1 \in \mathbb{R}^n$ is a non-zero vector.  Therefore, for  $j=1,2$, we have
\begin{align} \label{prod_mat_norm4_0} 
 \norm{M_{k+1}^j \ldots M_2^j M_1^j \mathbf{z}^j}_T < &\max \Bigg \{ \norm{\mathbf{z}_1^j}_1, \frac{1}{k+3} \Big ((k+1)\norm{\mathbf{z}_1^j}_1 + 2 \norm{\mathbf{z}_2^j}_1 \Big ), \ldots, \nonumber \\ &\frac{1}{T} \Big(  (k+1)\norm{\mathbf{z}_1^j}_1 + (T-k-1) \norm{\mathbf{z}_{T-k-1}^j}_1 \Big) \Bigg \}. 
\end{align}
Now, similar to part (i), we suppose that without loss of generality $\norm{\mathbf{z}_i^j}_1$ is the maximum value among all $\norm{\mathbf{z}_t^j}_1$, where $t=1,2,\ldots,i-1,i+1,\ldots,T$. Then, by definition of the norm $\norm{\cdot}_T$, we write $\norm{\mathbf{z}^j}_T = \norm{\mathbf{z}_i^j}_1$, for $j=1,2$. After replacing this in \eqref{prod_mat_norm4_0}, we get 
\begin{align*}
 \norm{M_{k+1}^j \ldots M_2^j M_1^j \mathbf{z}^j}_T < \norm{\mathbf{z}^j}_T, \quad \text{for } j=1,2. 
\end{align*} 
\end{itemize} \qedhere
\end{proof}
 Now, we proceed as follows to show that the product of matrices $U \in \mathcal{S}$ (cf. \eqref{def:mat_U}) is non-expansive. For $v \in \mathbb{N}$, let $t_{v} \in K$ be the (combined) capacity event time-instant. Furthermore, let us choose $k \in \mathbb{N}$, where $t_k \in K$, and let matrix $H_{k}$ denote the product of $k+1$ matrices $U \in \mathcal{S}$. Then starting from the $v^{\mathrm{\mathrm{th}}}$ capacity event time-instant $t_v$, we write that 
\begin{align} \label{mat_H_0}
\mathbf{H}_{k}(v+k+1) & = \mathbf{U}(v+k+1) \ldots \mathbf{U}(v).
\end{align}
Thus, for $k \in \mathbb{N}$ and $T \leq k+1$, we rewrite \eqref{Markov2} as follows,
\begin{align} \label{Markov_H}
\bm{\xi}(k+1) & = \mathbf{H}_{T-1}(k) \bm{\xi}(k-T+1).
\end{align}
Suppose that in $k+1$ capacity events, $k^1 \in \mathbb{N}$ capacity events of resource $\mathcal{R}^1$ occur, analogously, for $k^2 \in \mathbb{N}$, i.e.,  $t_k^1 \in K^1$; $t_k^2 \in K^2$; and $k^1 + k^2 = k+1$. Then we write the product of $k+1$ matrices $U \in \mathcal{S}$ (cf. \eqref{def:mat_U}) as
\begin{align} \label{mat_H}
H_{k} &= \begin{bmatrix}
		(D^1)^{k^1} & 0  \\
		0  & (D^2)^{k^2} 
\end{bmatrix}.
\end{align}
 For $j=1,2$, let $\Psi^j$ be the time between one capacity event of resource $\mathcal{R}^j$ when all agents back-off to the next capacity event. Recall that  $\alpha^j$ is the additive increase factor, $\beta^j$ is the multiplicative decrease factor, and $\mathcal{C}^j$ is the capacity of the resource $\mathcal{R}^j$; also, $n$ is the number of agents in the system. Then we calculate $\Psi^j$ as follows,
\begin{align*}
\Psi^j = \frac{(1-\beta^j)\mathcal{C}^j}{n \alpha^j}, \quad \text{ for } j=1,2.
\end{align*}
  Without loss of generality, we write that $0 < \Psi^1 \leq \Psi^2$. Let us choose $k \in \mathbb{N}_+$ such that $k \Psi^1 \geq \Psi^2$, then in $(k+1)$ capacity events, each resource has a capacity event with positive probability in which every agent backs-off.
Moreover, let in $k^1$ capacity events there occurs a capacity event with positive probability in which all the agents back-off for resource $\mathcal{R}^1$; analogously, for $k^2$. 
Now, let matrix $Y_{k}$ be the product of $k+1$ matrices $U \in \mathcal{S}$ (cf. \eqref{def:mat_U}) in which there exists a capacity event with positive probability for each resource when all the agents back-off. Thus, following the definition of $M_g^j$, for $g \in \{1,\ldots, k+1 \}$, as in \eqref{def:Mg}, the product $M_{k^j}^j \ldots M_1^j$ contains at least one $E^j \in \mathcal{Q}^j$ matrix, for all $j$. Hence, by definition of $U \in \mathcal{S}$, we write matrix $Y_{k}$ as 
\begin{align} \label{mat_Y}
Y_{k} &= \begin{bmatrix}
		M_{k^1}^1 M_{k^1-1}^1 \ldots M_1^1 & 0  \\
		0  & M_{k^2}^2 M_{k^2-1}^2 \ldots M_1^2 
\end{bmatrix}.
\end{align}
Now, let $\mathcal{V}$ be a finite set of the product of matrices $U \in \mathcal{S}$; furthermore, let $H \in \mathcal{V}$ denote the product of a finite number of $U \in \mathcal{S}$ matrices and $Y \in \mathcal{V}$ denote the product of a finite number of $U \in \mathcal{S}$ matrices where all agents back-off for both the resources. Then we have the following results.
\begin{lem} \label{lem:avg_contraction}
 \begin{itemize}
\item[(i)] Let $H \in \mathcal{V}$, and let $W$ be the subspace defined as 
\begin{align*}
W & \triangleq \{(\mathbf{z}^1, \mathbf{z}^2) \in \mathbb{R}^{2Tn} \vert \mathbf{e}^\top \mathbf{z}^1_{t} = \mathbf{e}^\top \mathbf{z}^2_{t} = 0, \\ &\text{ for } t=1,2,\ldots,T \}.
\end{align*}
 Then $W$ is an invariant subspace under all $H \in \mathcal{V}$.
\item[(ii)] For all $\mathbf{z} \in W$ and $H \in \mathcal{V}$, the norm $\norm{H \mathbf{z}} \leq \norm{\mathbf{z}}$ holds.
\item[(iii)] For non-zero $\mathbf{z} \in W$ and $Y \in \mathcal{V}$, the norm $\norm{Y \mathbf{z}} < \norm{\mathbf{z}}$ holds.
\end{itemize}
\end{lem}

\begin{proof} 
\begin{itemize}
	\item[(i)] Using the fact that for $j=1,2$, we have $\mathbf{e}^\top A^j = \mathbf{e}^\top$; proof is an easy consequence of this. 
\item[(ii)]  Let us choose fixed $k \in \mathbb{N}$ and $k^1, k^2 \in \mathbb{N}$, where $t_{k^1} \in K^1$ and $t_{k^2} \in K^2$, and $k^1+k^2=k+1$. Let matrices $H \in \mathcal{V}$ be the product of $U \in \mathcal{S}$ matrices for $k+1$ capacity events. Then, we write $H \in \mathcal{V}$ as \eqref{mat_H} and by definition of the norm $\norm{\cdot}$, we obtain 
\begin{align} \label{eq:norm_1}
 \norm{H \mathbf{z}} &= 
\norm{ \begin{bmatrix}
  (D^1)^{k^1} & 0  \\
  0  & (D^2)^{k^2}
\end{bmatrix} \begin{bmatrix}
  \mathbf{z}^1  \\
  \mathbf{z}^2
\end{bmatrix}
}. \nonumber\\
&= \max \Big \{\norm{ (D^1)^{k^1} \mathbf{z}^1 }_T, \norm{(D^2)^{k^2} \mathbf{z}^2}_T \Big\}.
  \end{align}
    From Lemma \ref{lem:prod_cont}, we write that $\norm{(D^j)^{k^j} \mathbf{z}^j }_T \leq \norm{\mathbf{z}^j}_T$, for $j=1, 2$; placing this value in \eqref{eq:norm_1}, we get
  \begin{align*}
\norm{H \mathbf{z}} \leq \max\{ \norm{\mathbf{z}^1}_T, \norm{\mathbf{z}^2}_T \}.
\end{align*}
Hence, by definition of $\norm{.}$, we have  
 $\norm{H \mathbf{z}} \leq \norm{\mathbf{z}}$. 
%
\item[(iii)]
Similar to (ii), we choose fixed $k \in \mathbb{N}$ and $k^1, k^2 \in \mathbb{N}$.  Now, suppose that in $k+1$ capacity events there exists a capacity event for each resource when all the agents back-off. Let the matrix $Y \in \mathcal{V}$ be the product of $k+1$ such matrices $U \in \mathcal{S}$. Furthermore, for $j=1,2$ and $g \in \{1,2,\ldots, k+1 \}$, the matrix $M_g^j$ is defined as \eqref{def:Mg}. Then, we write $Y \in \mathcal{V}$ as \eqref{mat_Y} to have
\begin{align*} 
 \norm{Y \mathbf{z}} &= 
\norm{ \begin{bmatrix}
  M_{k^1}^1 \ldots M_1^1 & 0  \\
  0  & M_{k^2}^2 \ldots M_1^2 
\end{bmatrix} \begin{bmatrix}
  \mathbf{z}^1  \\
  \mathbf{z}^2
\end{bmatrix}} \nonumber\\
&= \max \Big \{ \norm{ M_{k^1}^1 \ldots M_1^1 \mathbf{z}^1 }_T, \norm{ M_{k^2}^2 \ldots M_1^2 \mathbf{z}^2}_T \Big \}.
  \end{align*}
    From Lemma \ref{lem:prod_cont} (ii) and by definition of $\norm{\cdot}$, we get
 $\norm{Y \mathbf{z}} < \norm{\mathbf{z}}$. \qedhere
 \end{itemize} 
\end{proof}
Let $p_{H}(z)$ be the probability of choosing a particular matrix $H \in \mathcal{V}$, and $p_{Y}(z)$ be the probability of choosing a particular matrix $Y \in \mathcal{V}$ from the set $\mathcal{V}$. Recall that $\Sigma^{T}$ is the simplex in $\mathbb{R}^{Tn}$.
 Then, we present the following results.
\begin{thm} [Unique invariant \cite{Barnsley1988}] \label{Th:Barnsley}
For all $i$ and $j$, let the probability distribution $\lambda_i^j(\cdot)$ be Lipschitz continuous. Let $W$ be the invariant subspace under all $H \in \mathcal{V}$. Additionally, let there exist $\mu<1$ and $\eta>0$ such that the following holds
\begin{itemize}
\item[(i)] for all $\mathbf{z},\mathbf{w} \in\Sigma^{T} \times \Sigma^{T}$,
\begin{align} \label{Th:cont_avg}
\sum_{H \in \mathcal{V}} p_{H}(\mathbf{z}) \frac{ \norm{H(\mathbf{z}-\mathbf{w}) }}{\norm{\mathbf{z}-\mathbf{w}}} < 1, \text{ and,}
\end{align}  
\item[(ii)] for all $\mathbf{z},\mathbf{w} \in\Sigma^{T} \times \Sigma^{T}$, 
\begin{align} \label{Th:prod_prob}
&\sum_{H \in \mathcal{V}, \norm{H(\mathbf{z}-\mathbf{w})} \leq \mu \norm{\mathbf{z}-\mathbf{w}}} p_{H}(\mathbf{z}) p_{H}(\mathbf{w}) \geq \eta^2. 
\end{align}
\end{itemize}
Then, the model has an attractive probability distribution, and it has a unique invariant measure. Here, $\mathbf{z} - \mathbf{w} \in W$. 
\end{thm}

\begin{thm} [Ergodic property \cite{Elton1987}] \label{th:ergodicity}
Let $\pi^T$ be an invariant probability distribution on $\Sigma^{T} \times \Sigma^{T}$. If there is a contraction on average \eqref{Th:cont_avg}, and probabilities are bounded away from zero and are Lipschitz continuous. Then, for every initial value $\bm{\xi}(0) \in \Sigma^{T} \times \Sigma^{T}$, the following (ergodic property) holds,
\begin{align} \label{erg_pro}
\lim_{k \to \infty} \frac{1}{k+1} \sum_{\ell=0}^k \bm{\xi}(\ell) =  \mathbb{E}(\pi^T), \text{ almost surely}.
\end{align}
\end{thm}
\begin{thm} [Convergence to unique invariant measure]
For all $i$ and $j$, let the back-off probability $\lambda_i^j(\cdot)$ be Lipschitz continuous and strictly increasing in $[0,1]$; also, let there exist $\lambda_{min}^j>0$ such that $\lambda_i^j(\cdot) \geq \lambda_{min}^j$. Then, 
\begin{itemize}
\item[(i)] for every time window $T \geq 1$, there exists a unique invariant measure $\pi^T$ on $\Sigma^{T} \times \Sigma^{T}$, and,
\item[(ii)] for every $\bm{\xi}(0) \in \Sigma^{T} \times \Sigma^{T}$, the ergodic property \eqref{erg_pro} holds.
\end{itemize}
\end{thm}
\begin{proof}
\begin{itemize}
\item[(i)]
Recall that $\mathcal{V}$ is a set of the product of a finite number of $U \in \mathcal{S}$ matrices. Also, matrices $Y \in \mathcal{V}$ are the product of a finite number of $U \in \mathcal{S}$ matrices where there exists a capacity event for each resource in which all the agents back-off. Furthermore, given that $\lambda_{min}^j>0$ and $\lambda_i^j(\cdot) \geq \lambda_{min}^j$, for all $i$ and $j$. Thus, we write that $p_{Y}(\mathbf{z}) \geq \big(\lambda_{min}^1 \big)^n \big(\lambda_{min}^2 \big)^n > 0$, for all $\mathbf{z} \in \Sigma^{T} \times \Sigma^{T}$; hence, the probability of matrix $Y \in \mathcal{V}$ stays away from zero. 
Now, for all $\mathbf{z}$, $\mathbf{w} \in \Sigma^{T} \times \Sigma^{T}$ we have 
\begin{align} \label{eq:sumH}
&\sum_{H \in \mathcal{V}} p_{H}(\mathbf{z}) \norm{H(\mathbf{z} - \mathbf{w}) } = p_{Y}(\mathbf{z}) \norm{Y (\mathbf{z}-\mathbf{w}) } + \nonumber \\ \quad &\sum_{H \in \mathcal{V} - \{Y\}} p_{H}(\mathbf{z}) \norm{H (\mathbf{z} - \mathbf{w})}.
\end{align}
For all $\mathbf{z}, \mathbf{w} \in \Sigma^{T} \times \Sigma^{T}$, placing the values from Lemma \ref{lem:avg_contraction} in \eqref{eq:sumH}, we get
\begin{align} \label{eq:contrac2}
 &\sum_{H \in \mathcal{V}} p_{H}(\mathbf{z}) \norm{H (\mathbf{z} - \mathbf{w}) } < p_{Y}(\mathbf{z}) \norm{\mathbf{z} - \mathbf{w}} + \nonumber \\ & \sum_{H \in \mathcal{V} - \{ Y \}} p_{H}(\mathbf{z}) \norm{ \mathbf{z} - \mathbf{w}}.
\end{align} 
Since $\sum_{H \in \mathcal{V}} p_{H}(\mathbf{z}) = 1$; thus, $p_{Y}(\mathbf{z}) =1- \sum_{H \in \mathcal{V}-\{Y\}} p_{H}(\mathbf{z})$. Now, with little algebraic manipulation, from \eqref{eq:contrac2}, for all $\mathbf{z}, \mathbf{w} \in \Sigma^{T} \times \Sigma^{T}$, we obtain the contraction on average \eqref{Th:cont_avg}.

Now, since $p_{Y}(\mathbf{z}) \geq \big(\lambda_{min}^1 \big)^n \big(\lambda_{min}^2 \big)^n$, for $\mathbf{z} \in\Sigma^{T} \times \Sigma^{T}$; then for $\mu < 1$ and $\mathbf{z},\mathbf{w} \in\Sigma^{T} \times \Sigma^{T}$, we write that
\begin{align*} 
\sum_{H \in \mathcal{V}, \norm{ H (\mathbf{z} - \mathbf{w})} \leq \mu \norm{\mathbf{z} - \mathbf{w} }} p_{H}(\mathbf{z}) p_{H}(\mathbf{w}) \geq \big( \big(\lambda_{min}^1 \big)^n \big(\lambda_{min}^2 \big)^n \big)^2.
\end{align*}
Let $\eta$ denote $\big(\lambda_{min}^1 \big)^n \big(\lambda_{min}^2 \big)^n$, then for all $\mathbf{z},\mathbf{w} \in\Sigma^{T} \times \Sigma^{T}$, and $\mu<1$, we get \eqref{Th:prod_prob}.
 Therefore, all the conditions of Theorem \ref{Th:Barnsley} are satisfied; hence, we conclude that the model has attractive probability distribution and invariant measure. Because of the attractivity property, the invariant measure is unique. 
\item[(ii)] From part (i) of the proof, we know that the model has contraction on average \eqref{Th:cont_avg} and because probabilities stay away from zero and are Lipschitz continuous; the conditions of Theorem \ref{th:ergodicity} are satisfied. Hence, the ergodic property \eqref{erg_pro} holds.
 \qedhere
 \end{itemize}
\end{proof}
\section{Conclusion}  \label{conclusion}
Alam et al. \cite{Syed2018} proposed a multi-resource allocation algorithm which incurs minimal communication overhead and does not require inter-agent communication.  Moreover, we found the average allocation over a fixed window size and modeled the system as a Markov chain with place-dependent probabilities. We proved that the average allocations vector converges to a unique invariant measure asymptotically, and the ergodic property holds.
As future work, we prove the almost sure convergence of long-term average allocations to optimal allocations.
\bibliographystyle{IEEEtran}
\bibliography{Dist_op} 
\end{document}